\documentclass[a4paper,12pt,reqno]{amsart}
\usepackage{a4wide}
\usepackage{amstext}
\usepackage{ucs}
\usepackage[utf8x]{inputenc}
\usepackage{amsmath}
\usepackage{amsfonts}
\usepackage{amssymb}
\usepackage{amsthm}
\usepackage{caption}
\usepackage{subfigure}
\usepackage{upgreek}
\usepackage{xcolor}
\usepackage{mathtools}
\usepackage[UKenglish,USenglish]{babel}
\usepackage[T1]{fontenc}
\usepackage{etoolbox}
\usepackage{esint,commath,ifthen}
\usepackage{graphics}
\usepackage{enumitem}
\usepackage{epstopdf}
\usepackage{faktor}

\usepackage[normalem]{ulem}
\usepackage{todonotes}

\usepackage{hyperref}
\hypersetup{
    colorlinks=true,
    linkcolor=blue,
    citecolor=magenta
    }
\usepackage{mathrsfs}
\usepackage{tikz-cd}

\newtheorem{theorem}{Theorem}[section]
\newtheorem{lemma}[theorem]{Lemma}
\newtheorem{proposition}[theorem]{Proposition}

\theoremstyle{definition}
\newtheorem{definition}[theorem]{Definition}

\newtheorem{remark}[theorem]{Remark}

\newcommand{\R}{\mathbb{R}}
\newcommand{\Z}{\mathbb{Z}}
\newcommand{\C}{\mathbb{C}}
\newcommand{\aech}{\mathbb{H}}

\newcommand{\N}{\mathbb{N}}
\newcommand{\Q}{\mathbb{Q}}
\newcommand{\bs}{\boldsymbol}

\allowdisplaybreaks[4]

\title[Weighted badly approximable vectors and bounded orbits]{Weighted badly approximable complex vectors and bounded orbits of certain diagonalizable flows}
\author{Gaurav Sawant}
\address{Department of Mathematics, The Institute of Science, 15, Madam Cama Road, Mumbai 400032, India}
\email{gauravsawant.math@gmail.com}
\date{}

\begin{document}

\begin{abstract}
We show an analogue of a theorem of An, Ghosh, Guan, and Ly \cite{AGGL} on weighted badly approximable vectors for totally imaginary number fields.  We show that for $G=\mathrm{SL}_2(\C)\times\dots\times\mathrm{SL}_2(\C)$ and $\Gamma<G$ a lattice subgroup, the points of $G/\Gamma$ with bounded orbits under a one-parameter Ad-semisimple subgroup of $G$ form a hyperplane-absolute-winning set.  As an application, we also provide a generalization of a result of Esdahl-Schou and Kristensen \cite{ESK} about the set of badly approximable complex numbers.
\end{abstract}

\maketitle

\section{Introduction}
The study of badly approximable numbers through topological games has been driven by the seminal work of Schmidt \cite{Sch22,Sch23} who showed that the set of real numbers badly approximable by the rationals is thick.  The study of diophantine approximation over complex numbers also has a long history.  Approximation of complex numbers by ratios of elements in the ring of integers of the quadratic number field $\Q(\sqrt{-D})$ has been studied in \cite{ESK}, building on the work of \cite{ASchmidt} and \cite{DK}.  Esdahl-Schou and Kristensen \cite{ESK} showed that the set 
\begin{equation}\label{1_Eq7.1}
 \mathbf{Bad}(K)=\left\{ z\in\C:\left|z-\dfrac{p}{q} \right|\geq \dfrac{\kappa(z)}{|q|^2}~~\text{for~all}~p,q\in\mathcal{O}_{K},~q\neq 0 \right\}
\end{equation}
of complex numbers badly approximable by ratios of elements in the ring of integers $\mathcal{O}_{K}$ in $K$ for $K=\Q(\sqrt{-D})$ has full Hausdorff dimension, where $D\in\{1,2,3,7,11,19,43,67,163 \}$.  From Stark \cite{Stark}, these are precisely those values of $D$ for which $\mathcal{O}_{\Q(\sqrt{-D})}$ is a UFD.  See \cite{Hines} for a characterization of $\mathbf{Bad}(\Q(\sqrt{-D}))$ in terms of continued fractions for $D=1,2,3,7,11$; these are precisely the values of $D$ for which $\mathcal{O}_D$ is a norm-Euclidean domain.

It is natural to look for generalizations of \cite{ESK} to {totally imaginary} number fields, which is one of the main results of this paper.
\begin{theorem}\label{1_T7.1}
Let $K$ be a totally imaginary number field of degree $n$ over $\Q$.  Then $\mathbf{Bad}(K)$ has full Hausdorff dimension.
\end{theorem}
We will obtain Theorem \ref{1_T7.1} as a consequence of a general statement about weighted badly approximable complex vectors involving Schmidt games.  McMullen \cite{McM} introduced the notions of strong-winning and absolute-winning properties and showed that the set of badly approximable numbers is absolute-winning.  Given a number field $K$, there are various generalizations of $\mathbf{Bad}(K)$ in higher dimensions depending upon the real and complex embeddings of $K$.

Weighted badly approximable vectors have been extensively studied over the past few years by \cite{KL}, \cite{BFKRW}, \cite{EGL} among many others.  Schmidt \cite{Sch82} conjectured that any two sets of real two-dimensional weighted badly approximable vectors intersect nontrivially; it was proven in \cite{BPV} that these sets in fact have strong finite intersection property in the sense of full Hausdorff dimension.  Recently, \cite{Hattori} has explored relations between badly approximable systems of linear forms and geometry of products of symmetric spaces using \cite{Dani1}, \cite{EGL}, \cite{KL}. 

The theory of Diophantine approximation over number fields is closely related to homogeneous dynamics; see \cite{Ghosh} for a survey.  The work of Dani \cite{Dani1,Dani2} established a correspondence between badly approximable numbers and bounded orbits of diagonalizable flows on homogeneous spaces.  Analogues of McMullen's result over totally real number fields were studied in \cite{EGL,AGGL} where the concept of weighted approximation was employed to show that badly approximable sets are hyperplane-absolute-winning (HAW).  In this paper, we {extend these findings to the case of a totally imaginary number field $K$.  Denoting by $\mathbf{Bad}(K,\mathbf{r})$ the badly approximable complex vectors relative to $K$ with respect to the real weight vector $\mathbf{r}$ (See (\S3, Definition \ref{D3.2}) for a precise definition), we show that:
\begin{proposition}\label{PInt3.5}
 $\mathbf{Bad}(K,\mathbf{r})$ is hyperplane-absolute-winning (HAW) in $\C^n$.
\end{proposition}
The above proposition is proved using a direct extension of the technique employed in \cite{AGGL}.  This in turn allows us to} investigate an analogue of \cite[Theorem 1.1]{AGGL} for products of finitely many copies of $\mathrm{SL}_2(\C)$, stated as below:
\begin{theorem}\label{T1.1_AGGL}
 Let $G=\mathrm{SL}_2(\C)\times\dots\times\mathrm{SL}_2(\C)$ be a product of $n$ copies of $\mathrm{SL}_2(\C)$ and let $\Gamma$ be a lattice subgroup of $G$.  Then for any one-parameter $Ad-$semisimple subgroup $F=\{g_t:t\in\R \}$ of $G$, the set 
 \[E(F):=\{x\in G/\Gamma\,:\,Fx~\text{is~bounded}\} \]
 is Hyperplane-Absolute-Winning (HAW).
\end{theorem}
The proof of the above theorem is achieved by an appropriate variation of the method used in \cite{AGGL} as predicted in \cite[\S1, Remarks 1, 2]{AGGL}.  Theorem \ref{T1.1_AGGL} verifies a conjecture of An, Guan, and Kleinbock \cite{AGK} in the case of $G=\mathrm{SL}_2(\C)\times\dots\times\mathrm{SL}_2(\C)$, $\Gamma=\mathrm{SL}_2(\mathcal{O}_{K})$ with $K$ being a totally imaginary field over $\Q$.  This conjecture is already verified for $G=\mathrm{SL}_3(\R)$ and $\Gamma=\mathrm{SL}_3(\Z)$ in \cite{AGK} and for $G=\mathrm{SL}_2(\R)\times\dots\times\mathrm{SL}_2(\R)$ in \cite{AGGL}; see also \cite{KL} and \cite{EGL} for earlier results.

We closely follow the strategy employed in \cite{AGGL}.  In \S2, we describe the hyperplane-absolute and the hyperplane-potential variants of the complex manifold version of Schmidt's game introduced in \cite{Sch22}.  We consider the special case of Theorem \ref{T1.1_AGGL} with a totally imaginary field $K$ of degree $n$ over $\Q$, $G=\prod_{j=1}^n \mathrm{SL}_2(\C)$, $\Gamma=\mathrm{SL}_2(\mathcal{O}_{K})$, and $F$ the diagonal flow in the left-multiplicative action of $G$ on $G/\Gamma$ in \S3.  The set $E(F)$ in Theorem \ref{T1.1_AGGL} is identified with weighted badly approximable complex vectors by an appropriate analogue of Dani correspondence.  The main input in the proof of Theorem \ref{T1.1_AGGL} is the alternate formulation of the set of weighted badly approximable complex vectors obtained in \S4.  In \S5, we show that the HAW property is preserved by biholomorphisms of complex manifolds by proving a holomorphic analogue of Proposition 2.3(c) of \cite{BFKRW} and Lemma 2.2(4) of \cite{AGGL}.  We combine all of these arguments in \S6 to complete the proof of Theorem \ref{T1.1_AGGL} along the lines of \cite[\S5]{AGGL}.  In \S7, we prove Theorem \ref{1_T7.1} as a consequence of Theorem \ref{T1.1_AGGL}; this is facilitated by Lemma \ref{L7.1} which allows us to apply Theorem \ref{T1.1_AGGL} to the twisted diagonal embedding of $\mathbf{Bad}(K)$ into $\C^n$.

\section{Schmidt's game and its variants in \texorpdfstring{$\C^n$}{Cn}}
Schmidt's game \cite{Sch22,Sch23} is a two-player topological game played on a Euclidean space.  The players, A and B, choose closed balls $\{A_j\}_{j\in\N}$ and $\{B_j\}_{j\in\N}$, respectively, according to a predetermined rule, in succession so that $\{B_j\}_{j\in\N}$ form a nested sequence.  Player A wins the game if $\bigcap_{j\in\N} B_j$ intersects a prescribed target set $E$; else, Player B wins.

Two of the variants of Schmidt's game are the hyperplane-absolute (HA) game and the hyperplane-potential (HP) game.  The HA game was introduced in \cite{BFKRW} and explored further in \cite{KW17}, and the HP game was studied in \cite{FSU}, both for real Euclidean spaces $\R^n$.  We provide the $\C^n$ version of these games below:

\subsection{Hyperplane-absolute (HA) game in \texorpdfstring{$\C^n$}{Cn}}{\label{SecHAW}}
Consider the Euclidean space $\C^n$.  {The complex hyperplanes in $\C^n$ are affine subspaces of complex dimension $n-1$.  (If $\C^n$ is viewed as $\R^{2n}$, such a complex hyperplane is the intersection of two real affine hyperplanes.)} For each {complex} hyperplane $\mathcal{H}$ and a $\delta>0$, call {the set} 
\[ \mathcal{H}^{(\delta)}:=\{\mathbf{z}\in\C^n:\inf_{\mathbf{w}\in \mathcal{H}}\|\mathbf{z}-\mathbf{w}\|<\delta\}. \]
{the $\delta$-neighborhood of $\mathcal{H}$.}
For $0<\beta<\frac{1}{3}$, the $\beta$-hyperplane-absolute game has the following ruleset:
\begin{itemize}
 \item The game begins with Player B's initial closed ball $B_0\subset\C^n$ of radius $\rho=\rho_0$.
 \item In the $j^{\text{th}}$ move, Player B has chosen a closed ball $B_j$ of radius $\rho_j$, while Player A selects a hyperplane neighborhood $\mathcal{H}_j^{(\delta_j)}$ with $\delta_j\leq\beta\rho_j$.
 \item Player B must choose a closed ball $B_{j+1}\subset B_j\setminus \mathcal{H}_j^{(\delta_j)}$ of radius $\rho_{j+1}\geq\beta\rho_j$ as their $(j+1)^{\text{th}}$ move.
\end{itemize}
The above procedure generates a sequence of nested closed balls $\{B_j\}_{j\in\N}$.

A subset $E\subset\C^n$ is said to be $\beta$-hyperplane-absolute-winning (i.e. $\beta$-HAW) if Player A can ensure that $E\bigcap\left(\cap_{j\in\N}B_j\right)\neq\emptyset$ for every possible sequence of moves $\{B_j\}_{j\in\N}$ by Player B.  Furthermore, $E$ is called hyperplane-absolute-winning (i.e. HAW) if it is $\beta$-HAW for any $0<\beta<\frac{1}{3}$.

We refer to \cite{Wells} for the terminology regarding complex manifold theory.  The basic properties of $\beta$-HAW subsets and HAW subsets of $\C^n$ are similar to those for $\R^n$ described in (\cite{BFKRW},\cite{KW17}), which we list in the lemma below (see also \cite[Lemma 2.1]{AGGL}).

\begin{lemma}\label{L2.1} \hfill
 \begin{enumerate}
  \item A HAW subset of $\C^n$ is always $\frac{1}{2}$-winning.
  \item If $0<\beta'\leq\beta<\frac{1}{3}$, then any $\beta'$-HAW set is also $\beta$-HAW.
  \item If $\{E_j\}_{j\in\N}$ are HAW sets, then so is $\bigcap_{j\in\N}E_j$.
 \end{enumerate}
\end{lemma}

Similar to extending the idea of HAW sets to the corresponding subsets of $C^1$ manifolds done in \cite{KW17}, one can describe HAW subsets of complex manifolds as follows:
\begin{itemize}
 \item Given an open subset $\Omega\subset\C^n$, one defines the HA game on $\Omega$ with the added condition that Player B makes their initial move ($B_0$) within $\Omega$.
 \item For a complex $n$-manifold $M$ with a holomorphic system $\{(U_\alpha,\phi_\alpha)\}$ giving a complex structure to $M$, a subset $E\subset M$ is called HAW on $M$ if $\phi_\alpha(E\cap U_\alpha)$ is HAW on $\phi_\alpha(U_\alpha)$ for each $\alpha$.
\end{itemize}
We can list the fundamental properties of HAW subsets of complex manifolds as below (see also \cite[Lemma 2.2]{AGGL}).
\begin{lemma}\label{L2.2} \hfill
 \begin{enumerate}
  \item HAW subsets of complex manifolds have full Hausdorff dimension.
  \item If $\{E_j\}_{j\in\N}$ are HAW subsets of a complex manifold, then so is $\bigcap_{j\in\N} E_j$.
  \item Let $M$ be a complex manifold with an open cover $\{U_\alpha\}$.  Then $E\subset M$ is HAW if and only if $E\cap U_\alpha$ is HAW in $U_\alpha$ for each $\alpha$.
  \item Let $\{M_j\}_{j=1}^n$ be a finite collection of complex manifolds, and let $E_j\subset M_j$ be HAW for each $1\leq j\leq n$.  Then $\prod_{j=1}^nE_j$ is HAW in $\prod_{j=1}^n M_j$.
 \end{enumerate}
\end{lemma}

{\begin{remark}
                  The main difference between the properties of HAW subsets of real Euclidean spaces (resp. manifolds) and their complex analogues comes from the necessity of preserving the underlying real or complex structures.  While the images of real HAW sets under $C^1$ diffeomorphisms are HAW as mentioned in \cite[Lemma 2.1(4), Lemma 2.2(3)]{AGGL}, one requires a biholomorphic correspondence in the complex case.  The analogous results for the complex case are proved in \S5.
                 \end{remark}
}

\subsection{Hyperplane-potential (HP) game in \texorpdfstring{$\C^n$}{Cn}}
The HP game is a two-parameter variant of Schmidt's game.  The real Euclidean space version of it, studied in \cite{FSU}, can be easily extended to complex Euclidean spaces $\C^n$ as follows:
\begin{itemize}
 \item Fix the two parameters $0<\beta<1$ and $\gamma>0$.
 \item Player B begins with a closed ball $B_0\subset\C^n$ of radius $\rho=\rho_0$.
 \item After Player B makes the $j^{\text{th}}$ move $B_j$, Player A chooses a countable collection of hyperplane neighborhoods $\{\mathcal{H}^{(\delta_{j,k})}_{j,k}:k\in\N\}$ that satisfy 
 \[ 
 \sum_{k\in\N} \delta_{j,k}^\gamma\leq(\beta\rho_j)^\gamma. 
 \]
 
 \item For their $(j+1)^{\text{th}}$ move, Player B chooses a closed ball $B_{j+1}\subset B_j\setminus\left(\bigcup_{k\in\N} \mathcal{H}^{(\delta_{j,k})}_{j,k} \right)$ of radius $\rho_{j+1}\geq\beta\rho_j$.
\end{itemize}
The above procedure generates a nested sequence of closed balls $\{B_j\}_{j\in\N}$.

A subset $E\subset\C^n$ will be called $(\beta,\gamma)$-hyperplane-potential-winning (that is, $(\beta,\gamma)$-HPW) if Player A can ensure that \[ \bigcap_{j=0}^\infty B_j\cap\left( E\cup\bigcup_{j=0}^\infty \bigcup_{k\in\N} \mathcal{H}_{j,k}^{(\delta_{j,k})}\right) \neq\emptyset\]
for every possible sequence of moves $\{ B_j \}_{j\in\N}$ by Player B.  Furthermore, $E$ is called hyperplane-potential-winning (HPW) if it is $(\beta,\gamma)$-HPW for any $0<\beta<1$ and $\gamma>0$.

{\begin{remark}
                  In view of \cite[Theorem C.8]{FSU}, the notions of HAW and HPW sets are equivalent.  Thus to prove that a particular set is hyperplane-absolute-winning, one can set up a suitable hyperplane-potential game and show that the set is winning with respect to that game.
                 \end{remark}
}

\section{A special case of Theorem \ref{T1.1_AGGL}}\label{S3}
Let $K$ be a totally imaginary field of degree $n$ over $\Q$.  If $\Sigma$ denotes the set of distinct field embeddings $K\hookrightarrow\C$, then $|\Sigma|=n$.  The images of any $p\in K$ are listed as a vector in $\C^n$ using the map
\begin{equation}\label{Theta} \Theta:K\to\prod_{\sigma\in\Sigma}\C, \quad p\mapsto(\sigma(p))_{\sigma\in\Sigma}.
\end{equation}
From \cite{BHC,Morris}, it is known that the group $\text{Res}_{K/\Q}\mathrm{SL}_2(\Z[i])$ is a lattice in $\text{Res}_{K/\Q}\mathrm{SL}_2(\C)= \prod_{\sigma\in\Sigma} \mathrm{SL}_2(\C)$, where $\text{Res}_{K/\Q}$ denotes the Weil restriction-of-scalars functor.  We set
\[G=\prod_{\sigma\in\Sigma} \mathrm{SL}_2(\C), \quad \Gamma=\text{Res}_{K/\Q}\mathrm{SL}_2(\Z[i]). \]
From the definitons, we see that $\text{Res}_{K/\Q}\mathrm{SL}_2(\Z[i])$ coincides with the subgroup $\Theta(\mathrm{SL}_2(\mathcal{O}_{K}))$, where $\Theta(g)=(\sigma(g))_{\sigma\in\Sigma}$.  We will now state the following special (weighted) case of Theorem  \ref{T1.1_AGGL}:

\begin{proposition}\label{P3.1}
 Let $\mathbf{r}\in\R^n$ be a real weight vector with $r_\sigma\geq 0$ for $\sigma\in\Sigma$ and $\sum_{\sigma\in\Sigma}r_\sigma=1$.  Define the vector diagonal flow
 \begin{equation}\label{eq3.1}
  g_{\mathbf{r}}(t):=\left( \begin{pmatrix}
                         e^{r_\sigma t} & 0 \\ 0 & e^{-r_\sigma t}
                        \end{pmatrix}
 \right)_{\sigma\in\Sigma}
 \end{equation}
and let $F_{\mathbf{r}}=\{g_{\mathbf{r}}(t):t\in\R\}$ be the full orbit of $g_{\mathbf{r}}(0)=\left( (I_2) \right)_{\sigma\in\Sigma}$.  Then the set \[ E(F_{\mathbf{r}}):=\{x\in G/\Gamma:F_{\mathbf{r}}x~\text{is~bounded}\} \] is HAW.
\end{proposition}
We will prove Proposition \ref{P3.1} by studying the bounded forward orbits \[ E(F_{\mathbf{r}}^+):=\{x\in G/\Gamma:F_{\mathbf{r}}^+ x~\text{is~bounded}\}, \] where $F_{\mathbf{r}}^+=\{g_{\mathbf{r}}(t):t\geq 0 \}$.  We will closely follow the proof of \cite[Proposition 3.1]{AGGL}.

Fix the weight vector $\mathbf{r}$.  Let $\Sigma_+$ and $\Sigma_0$ denote the collection of positive and zero coordinates, respectively, of $\mathbf{r}$.  Let $|\Sigma_+|=n_+$.  Then $|\Sigma_0|=n-n_+$.  Moreover, let $\omega\in\Sigma$ correspond to the largest coordinate of $\mathbf{r}$: $r_\omega=r=\max_{\sigma\in\Sigma}r_\sigma$.  Also, define the $\mathbf{r}$-norm on $\C^n$ corresponding to the fixed weight vector $\mathbf{r}$ by
\[ \|\mathbf{z}\|_{\mathbf{r}}:=\max_{\sigma\in\Sigma_+}|z_\sigma|^{\frac{1}{r_\sigma}}. \]
Each $p\in K$, as embedded in $\prod_{\sigma\in\Sigma}\C$ via the map $\Theta$, would then have the $\mathbf{r}$-norm
\[ \|p\|_{\mathbf{r}}=\|\Theta(p)\|_{\mathbf{r}}=\max_{\sigma\in\Sigma_+}|\sigma(p)|^{\frac{1}{r_\sigma}}. \]

\begin{definition}\label{D3.2}
 We call $\mathbf{z}=(z_\sigma)_{\sigma\in\Sigma}\in\prod_{\sigma\in\Sigma}\C$ a $(K,\mathbf{r})$-badly approximable vector if
 \[ \inf_{\overset{\scriptstyle p\in\mathcal{O}_{K}}{q\in\mathcal{O}_{K}\setminus\{0\}}} \max\left\{  \max_{\sigma\in\Sigma_+}\|q\|_{\mathbf{r}}^{r_\sigma}|\sigma(q)z_\sigma+\sigma(p)|, \max_{\sigma\in\Sigma_0}\max\{|\sigma(q)z_\sigma+\sigma(p)|,|\sigma(q)|\}\right\}>0.\]
 We denote by $\mathbf{Bad}(K,\mathbf{r})$ the set of $(K,\mathbf{r})$-badly approximable vectors.
\end{definition}
\begin{remark}
 The above is the complex analogue of \cite[Definition 3.2]{AGGL}.
\end{remark}
Let $\displaystyle H=\left( \begin{pmatrix}
                             1 & * \\ 0 & 1
                            \end{pmatrix}
 \right)_{\sigma\in\Sigma}$ be the strictly upper triangular subgroup of $G$ that is an embedding of $\prod_{\sigma\in\Sigma}\C$ via the map
 \[ \iota: \prod_{\sigma\in\Sigma}\C\to\prod_{\sigma\in\Sigma}\mathrm{SL}_2(\C), \quad \iota((z_\sigma)_{\sigma\in\Sigma})= \left( \begin{pmatrix}
                             1 & z_\sigma \\ 0 & 1
                            \end{pmatrix}
 \right)_{\sigma\in\Sigma}. \]
 
The bounded $F_{\mathbf{r}}^+$ trajectories and vectors in $\mathbf{Bad}(K,\mathbf{r})$ are related to each other by the {following version of} Dani correspondence:
\begin{proposition}\label{P3.4}
 A vector $\mathbf{z}=(z_\sigma)_{\sigma\in\Sigma}$ is in $\mathbf{Bad}(K,\mathbf{r})$ if and only if the trajectory $F_{\mathbf{r}}^+\iota(\mathbf{z})\Gamma$ is bounded in $G/\Gamma$; that is, if $\pi:G\to G/\Gamma$ is the natural projection, then
 \begin{equation}\label{E3.2}
  \mathbf{Bad}(K,\mathbf{r})=\iota^{-1}(\pi^{-1}(E(F_{\mathbf{r}}^+))\cap H).
 \end{equation}
\end{proposition}
\begin{proof}
 Let $\Sigma=\{\sigma_1,\dots,\sigma_n\}$ and $r_{\sigma_j}=r_j$.  Without loss of generality, assume that $r_j>0$ for $1\leq j\leq n_+$ so that $r_j=0$ for $n_+<j\leq n$.  If $D_{K}$ denotes the discriminant of $K$, then $D_{K}^{-\frac{1}{2n}}\Theta(\mathcal{O}_{K})$ is a unimodular lattice of $\C^n$.  Let $L_{K}:=D_{K}^{-\frac{1}{2n}}\Theta(\mathcal{O}_{K})\times D_{K}^{-\frac{1}{2n}}\Theta(\mathcal{O}_{K}) \subset \C^{2n}$, and define a homomorphism $\psi:G\to\mathrm{SL}_{2n}(\C)$ by
 \begin{equation}\label{E3.3}
  \psi(g)=\begin{pmatrix}
           \text{diag}(a_{1,1},\dots,a_{1,n}) & \text{diag}(a_{2,1},\dots,a_{2,n}) \\ \text{diag}(a_{3,1},\dots,a_{3,n}) & \text{diag}(a_{4,1},\dots,a_{4,n})
          \end{pmatrix} \; \text{for}~g=\left( \begin{pmatrix}
          a_{1,j} & a_{2,j} \\ a_{3,j} & a_{4,j}
          \end{pmatrix}\right)_{1\leq j \leq n}.
 \end{equation}
Observe that $\psi(g)$ fixes $L_{K}$ for any $g\in\Gamma$.  Now suppose that $\psi(g) L_{K}=L_{K}$ for some $g$ as in \eqref{E3.3}.  It follows that $a_{2l+1,j}\sigma_j(k)+a_{2l,j}\sigma_j(k')\in\sigma_j(\mathcal{O}_{K})$ for all $1\leq j \leq n$, $l=0,1$, $k,k'\in\mathcal{O}_{K}$.  Choosing $k$ or $k'$ to be 0, it can be shown that $a_{l,j}\sigma_j(\mathcal{O}_{K})=\sigma_j(\mathcal{O}_{K})$ for all $1\leq j\leq n$, $1\leq l \leq 4$.  As a result,
\[ (g)_j\in M_{2\times 2}\left(\sigma_j(\mathcal{O}_{K})\right)\cap\mathrm{SL}_2(\C)=\sigma_j(\mathcal{O}_{K})~\text{for}~1\leq j \leq n. \]
Since an element of $\mathrm{SL}_2(\C)$ is uniquely determined by its action on $\widehat{\C}=\partial \aech^3$ via M\"{o}bius transformations, it follows that, if $\psi(g) L_{K}=L_{K}$, then $(g)_{j'}=\sigma_{j'}(\sigma_j^{-1}((g)_j))$, showing that $\psi(g)\in\Theta(\mathrm{SL}_2(\mathcal{O}_{K}))=\Gamma$.  We have thus proved
\begin{equation}\label{E3.4}
\{ g\in G:\psi(g) L_{K}=L_{K} \}=\Gamma. 
\end{equation}
As $\Gamma$ is a lattice in $G$, from \cite[Theorem 1.13]{MSR} and \eqref{E3.4} and implicitly using the fact that $\mathrm{SL}_{2n}(\C)/\mathrm{SL}_{2n}(\Z[i])$ is the space of unimodular lattices in $\C^{2n}$, we deduce that the embedding
\[ \phi:G/\Gamma\to\mathrm{SL}_{2n}(\C)/\mathrm{SL}_{2n}(\Z[i]),\quad\phi(g\Gamma)=\psi(g)L_{K} \] is a proper map.  Hence it follows that
\begin{equation}\label{E3.5}
 F_{\mathbf{r}}^+\iota(\mathbf{z})\Gamma~\text{is~bounded~in}~G/\Gamma\iff\psi(F_{\mathbf{r}}^+\iota(\mathbf{z}))L_{K}~\text{is~bounded~in}~\mathrm{SL}_{2n}(\C)/\mathrm{SL}_{2n}(\Z[i]).
\end{equation}
We have
\[\psi(g_{\mathbf{r}}(t))=\text{diag}(e^{r_1t},\dots,e^{r_nt},e^{-r_1t},\dots,e^{-r_nt}), \]
and \[ \psi(\iota(\mathbf{z}))=\begin{pmatrix}
                            I_n & \text{diag}(z_1,\dots,z_n) \\ 0 & I_n
                           \end{pmatrix}.
 \]
 Using Mahler's compactness criterion and \eqref{E3.5}, the rest of the proof follows as in \cite[Proposition 3.4]{AGGL}.
\end{proof}

\section{An alternate formulation of \texorpdfstring{$\mathbf{Bad}(K,\mathbf{r})$}{Bad(K,r)}}
In this section, we give a description of $\mathbf{Bad}(K,\mathbf{r})$, defined in \S3, analogous to \cite[\S4]{AGGL} and prove a variant of \cite[Proposition 3.5]{AGGL}.  Our treatment directly follows \cite[\S3,\S4]{AGGL}.  We include the details for convenience of the reader.

For $\varepsilon>0$, we set
\[\mathcal{O}_{K}(\mathbf{r},\varepsilon):=\{ q\in\mathcal{O}_{K}\setminus\{0\}:\max_{\sigma\in\Sigma_0}|\sigma(q)|\leq\varepsilon \}. \]
For $(p,q)\in\mathcal{O}_{K}\times\mathcal{O}_{K}(\mathbf{r},\varepsilon)$, define
\begin{equation}\label{delta}
 \Delta_\varepsilon(p,q):=\prod_{\sigma\in\Sigma_+}\overline{B}\left( \frac{\sigma(p)}{\sigma(q)},\frac{\varepsilon}{|\sigma(q)|\|q\|_{\mathbf{r}}^{r_\sigma}} \right) \times \prod_{\sigma\in\Sigma_0}\overline{B}\left( \frac{\sigma(p)}{\sigma(q)},\frac{\varepsilon}{|\sigma(q)|} \right) \subset\prod_{\sigma\in\Sigma}\C,
 \end{equation}
where $\overline{B}(z_0,\rho)=\{z\in\C:|z-z_0|\leq\rho\}$.  {In other words, the closed boxes $\Delta_\varepsilon(p,q)$ contain all points $\mathbf{z}=(\sigma(z)))_{\sigma\in\Sigma}$ whose coordinates are sufficiently close to those of $\Theta\left( \dfrac{p}{q}\right)$, where $\Theta$ is given by \eqref{Theta}.  Note that these sets as viewed in $\C^n$ have ``disc faces'', whereas such sets in $\R^{2n}$ would have rectangular faces.} Setting 
\begin{equation}\label{E4.1}
\mathbf{Bad}_\varepsilon(K,\mathbf{r}):=\prod_{\sigma\in\Sigma}\C\setminus\bigcup_{\overset{\scriptstyle p\in\mathcal{O}_{K}}{q\in\mathcal{O}_{K}\setminus\{0\}}}\Delta_\varepsilon(p,q),
\end{equation}
we have:
\begin{lemma}\label{L4.1}
 $\displaystyle \mathbf{Bad}(K,\mathbf{r}) = \bigcup_{\varepsilon>0} \mathbf{Bad}_\varepsilon(K,\mathbf{r}) $.
\end{lemma}
\begin{proof}
 It is enough to show that $\mathbf{Bad}_\varepsilon(K,\mathbf{r})$ is in fact the set of vectors $\mathbf{z}=(z_\sigma)_{\sigma\in\Sigma}\in\prod_{\sigma\in\Sigma}\C$ that satisfy
 \begin{equation}\label{E4.2}
  \inf_{\overset{\scriptstyle p\in\mathcal{O}_{K}}{q\in\mathcal{O}_{K}\setminus\{0\}}} \max\begin{Bmatrix}
\underset{1\leq j\leq n_+}{\max}\|q\|_{\mathbf{r}}^{r_j}|\sigma_j(q)z_j+\sigma_j(p)| , \\ 
\underset{n_+< j\leq n}{\max}\max\{|\sigma_j(q)z_j+\sigma_j(p)|,|\sigma_j(q)| \}                                                              
\end{Bmatrix} >\varepsilon
 \end{equation}
using the notation from the proof of Proposition \ref{P3.4}.  By the definition of $\mathcal{O}_{K}(\mathbf{r},\varepsilon)$, equation \eqref{E4.2} is equivalent to
\begin{equation}\label{E4.3}
 \inf_{\overset{\scriptstyle p\in\mathcal{O}_{K}}{q\in\mathcal{O}_{K}\setminus\{0\}}} \max\begin{Bmatrix}
\underset{1\leq j\leq n_+}{\max}\|q\|_{\mathbf{r}}^{r_j}|\sigma_j(q)z_j+\sigma_j(p)| , \\ 
\underset{n_+< j\leq n}{\max}\{|\sigma_j(q)z_j+\sigma_j(p)| \}                                                              
\end{Bmatrix} >\varepsilon.
\end{equation}
From the definition of $\Delta_\varepsilon(p,q)$, it is clear that if $\mathbf{z}\in\Delta_\varepsilon(p,q)$ for some $(p,q)\in\mathcal{O}_{K}\times\mathcal{O}_{K}(\mathbf{r},\varepsilon)$, then each coordinate $z_\sigma$ of $\mathbf{z}$ satisfies
\[ \text{either}~\|q\|_{\mathbf{r}}^{r_\sigma}|\sigma(q)z_\sigma+\sigma(p)|\leq\epsilon\quad\text{or}~ |\sigma(q)z_\sigma+\sigma(p)| \leq\varepsilon \]
depending on whether $\sigma\in\Sigma_+$ or $\sigma\in\Sigma_0$, respectively, implying that $\mathbf{z}\notin\mathbf{Bad}(K,\mathbf{r})$.  On the other hand, if $\mathbf{z}\notin\Delta_\varepsilon(p,q)$ for every $(p,q)\in\mathcal{O}_{K}\times\mathcal{O}_{K}(\mathbf{r},\varepsilon)$, then $\mathbf{z}$ automatically satisfies the equation in Definition \ref{D3.2}, and so $\mathbf{z}\in\mathbf{Bad}(K,\mathbf{r})$.  This proves the claimed equivalence.
\end{proof}

We now play a two-player $(\beta,\gamma)$-HP game in $\C^n$, where $0<\beta<\frac{1}{3}$ and $\gamma>0$.  We relabel the move index in such a way that the radius of the initial ball for Player B satisfies $\rho_0=\rho(B_0)<1$.  Fix an $R>0$ satisfying
\begin{equation}\label{E4.4}
 \frac{n}{R^\gamma-1}\leq\left(\frac{\beta^2}{2} \right)^\gamma,
\end{equation}
and set
\begin{equation}\label{E4.5}
 \varepsilon=\frac{\rho_0}{4R^{4n}}\quad\text{and}\quad H_l=\frac{\varepsilon R^l}{\rho_0} \quad(l\geq1).
\end{equation}
For each index $l\geq 0$, define
\begin{equation}\label{balls}\mathscr{B}_l:=\left\{ B\subset B_0:\frac{\beta\rho_0}{R^l}<\rho(B)\leq\frac{\rho_0}{R^l} \right\}. \end{equation}
Moreover, define the height function
\[ H:\mathcal{O}_{K}(\mathbf{r},\varepsilon)\to\R,\quad H(q)=\max_{\sigma\in\Sigma_+}|\sigma(q)|\|q\|_{\mathbf{r}}^{r_\sigma}. \]
Also, let
\begin{equation}\label{subd}
\mathscr{P}_m=\{q\in\mathcal{O}_{K}(\mathbf{r},\varepsilon):H_m\leq H(q)<H_{m+1}\},
\end{equation}
and
\begin{equation}\label{rsub} \mathscr{P}_{m,l}=\{q\in\mathscr{P}_m:H_mR^{4n(l-1)}\leq\|q\|_{\mathbf{r}}^{2r}<H_mR^{4nl}\}. \end{equation}
We have the following properties of $\mathcal{O}_K(\mathbf{r},\varepsilon)$:
\begin{lemma}
With the above notation, we have
\begin{enumerate}
 \item \label{L4.2} $1\leq\|q\|_{\mathbf{r}}^{\frac{1}{n}}\leq H(q)\leq\|q\|_{\mathbf{r}}^{2r}$ for all $q\in\mathcal{O}_{K}(\mathbf{r},\varepsilon)$.
\item \label{L4.3}
 $\displaystyle \mathcal{O}_{K}(\mathbf{r},\varepsilon) = \bigcup_{m\geq 0}\bigcup_{l\geq 1} \mathscr{P}_{m+l,l}$.
\end{enumerate}
\end{lemma}
\begin{proof}
Lemma 4.2\eqref{L4.2} follows analogously to \cite[Lemma 4.2]{AGGL}, and Lemma 4.2\eqref{L4.3} follows in a similar way to \cite[Lemma 4.3]{AGGL}. 
\end{proof}

The refined subdivisions \eqref{rsub} and the classes $\mathscr{B}_m$ as in \eqref{balls} have the following relation: 
\begin{lemma}\label{L4.4}
 Let $B\in\mathscr{B}_m$ and $l\in\N$.  Then the elements of the set \[ \mathscr{P}_{m+l,l}(B):=\{(p,q):q\in\mathscr{P}_{m+l,l} ~\text{and}~\Delta_\varepsilon(p,q)\cap B \neq\emptyset\} \] have a constant coordinate ratio.
\end{lemma}

Most of the proof of Lemma \eqref{L4.4} follows analogously to \cite[Lemma 4.4]{AGGL}.  We explain the key difference below:

\begin{proof}
 For any $B\in\mathscr{B}_m$ and $q\in\mathscr{P}_{m+l,l}$, \eqref{E4.5} implies
 \begin{equation}\label{E4.9}
  \rho(B)\leq\frac{\varepsilon R^{l+1}}{H(q)}.
 \end{equation}
Suppose that there are two pairs $(p_1,q_1)$ and $(p_2,q_2)$  with
\begin{equation}\label{E4.11}
 \Delta_\varepsilon(p_j,q_j)\cap B\neq\emptyset\quad(j=1,2)
\end{equation}
and
\begin{equation}\label{E4.10}
 p_1q_2\neq p_2q_1.
\end{equation}
Discreteness of the ring $\mathcal{O}_K$ provides a constant $c_K>0$ such that $N(p_1q_2-p_2q_1)\geq c_K$, so that
\begin{equation}\label{E4.12}
 \left| \prod_{\sigma\in\Sigma} \left( \frac{\sigma(p_1)}{\sigma(q_1)}-\frac{\sigma(p_2)}{\sigma(q_2)} \right) \right| \geq \left| \frac{N(p_1q_2-p_2q_1)}{N(q_1q_2)} \right| \geq \frac{c_K}{|N(q_1q_2)|}.
\end{equation}
Choosing $R$ large enough so that it satisfies both \eqref{E4.4} and $4^n\varepsilon^{2n}<c_K$, where $\varepsilon$ is as in \eqref{E4.5},  the contradictory inequality
\begin{equation}\label{E4.13}
 \left| \prod_{\sigma\in\Sigma} \left( \frac{\sigma(p_1)}{\sigma(q_1)}-\frac{\sigma(p_2)}{\sigma(q_2)} \right) \right| < \frac{c_K}{|N(q_1q_2)|}
\end{equation}
 also holds true, {proceeding on similar lines as in the proof of \cite[Lemma 4.4]{AGGL}}.  Hence any pairs $(p_j,q_j)$ satisfying \eqref{E4.11} must have a fixed ratio.
\end{proof}
{With the key lemma \ref{L4.4} at our disposal, the complex version of} \cite[Proposition 3.5]{AGGL} is as follows:
\begin{proposition}\label{P3.5}
 $\mathbf{Bad}(K,\mathbf{r})$ is hyperplane-absolute-winning (HAW) in $\C^n$.
\end{proposition}
 \begin{proof}
 The proof is identical to \cite[Proposition 3.5]{AGGL} with the use of \cite[Lemma 4.4]{AGGL} replaced by Lemma \ref{L4.4} above.
 \end{proof}

\begin{proof}[Proof of Proposition \ref{P3.1}]
Using Proposition \ref{P3.5}, this follows identically to the proof of \cite[Proposition 3.1]{AGGL}.
\end{proof}

\section{Images of HAW sets in complex manifolds under biholomorphisms}{\label{Biholo}}
In this section, we obtain a complex analogue of \cite[Theorem 2.4]{BFKRW}:

\begin{proposition}{\label{L214.223}}\hfill
 \begin{enumerate}
  \item Let $\phi:\C^n\to\C^n$ be a biholomorphism.  Then $\phi(E)$ is HAW whenever $E$ is.
  \item If $\phi:M\to N$ is a biholomorphism between complex manifolds, then $\phi(E)$ is a HAW subset of $N$ whenever $E\subset M$ is HAW.
 \end{enumerate}
\end{proposition}
\begin{proof}
\begin{enumerate}
 \item \textbf{Step 1}: We first show that the local distortion of $\phi$ is bounded.  For this, let $J_{\mathbf{z}}:=d_{\mathbf{z}}\phi$ be the derivative of $\phi$ at $\mathbf{z}$.  For any closed ball $B\subset\C^n$ and any $\bs{\xi},\bs{\eta}\in B$, we have
 \[ \| \phi(\bs{\xi}) - \phi(\bs{\eta})\|\leq\sup_{\mathbf{z}\in B}\|J_{\mathbf{z}}\|_{\mathrm{op}}\|\bs{\xi}-\bs{\eta}\|, \]
 where $\|\cdot\|_{\mathrm{op}}$ is the operator norm.  Let the linear approximations to $\phi$, $\phi^{-1}$ at $\bs{\xi}_0\in B$, $\bs{\eta}_0\in\phi(B)$, respectively, be $L_{\bs{\xi}_0}(\bs{\xi})=\phi(\bs{\xi}_0)+J_{\bs{\xi}_0}(\bs{\xi}-\bs{\xi}_0)$, $\widetilde{L}_{\bs{\eta}_0}(\bs{\eta})=\phi^{-1}(\bs{\eta}_0)+J_{\bs{\eta}_0}^{-1}(\bs{\eta}-\bs{\eta}_0)$.  We claim that
 \begin{equation}{\label{Eq3.1}}
  \sup_{\bs{\eta},\bs{\eta}_0\in\phi(B)}\dfrac{\|\widetilde{L}_{\bs{\eta}_0}(\bs{\eta})-\phi^{-1}(\bs{\eta})\|}{\|\bs{\eta}-\bs{\eta}_0\|}\to 0 \quad\text{as}~~\rho(B)\to 0.
 \end{equation}
For $\bs{\zeta}\in\phi(B)$, define $h_{\bs{\eta}_0}(\bs{\zeta})=\phi^{-1}(\bs{\zeta})-\widetilde{L}_{\bs{\eta}_0}(\bs{\zeta})$, so that $h_{\bs{\eta}_0}$ is differentiable with $d_{\bs{\zeta}}h_{\bs{\eta}_0}=J_{\bs{\zeta}}^{-1}-J_{\bs{\eta}_0}^{-1}$.  Since $\phi$ is biholomorphic, it is locally nonsingular (that is, $J_{\mathbf{z}}$ has full rank in a neighborhood of $\mathbf{z}$), so $\sup_{\bs{\zeta},\bs{\eta}_0\in\phi(B)}\|d_{\bs{\zeta}}h_{\bs{\eta}_0}\|\to 0$ as $\rho(B)\to 0$.  But
\begin{align*}
 \dfrac{\|\widetilde{L}_{\bs{\eta}_0}(\bs{\eta})-\phi^{-1}(\bs{\eta})\|}{\|\bs{\eta}-\bs{\eta}_0\|} &= \dfrac{\|\widetilde{L}_{\bs{\eta}_0}(\bs{\eta})-\widetilde{L}_{\bs{\eta}_0}(\bs{\eta}_0)+\phi^{-1}(\bs{\eta}_0)-\phi^{-1}(\bs{\eta})\|}{\|\bs{\eta}-\bs{\eta}_0\|} \\
 &= \dfrac{\|h_{\bs{\eta}_0}(\bs{\eta}_0)-h_{\bs{\eta}_0}(\bs{\eta})\|}{\|\bs{\eta}-\bs{\eta}_0\|} \leq \sup_{\bs{\zeta}\in\phi(B)}\|d_{\bs{\zeta}}h_{\bs{\eta}_0}\|_{\mathrm{op}},
\end{align*}
which proves the claim \eqref{Eq3.1}.

\textbf{Step 2}: Consider two $HA$ games $\mathcal{G},\mathcal{G}'$ in $\C^n$, where $\mathcal{G}'$ has target set $E$ and parameter $\beta'$, while $\mathcal{G}$ has target set $\phi^{-1}(E)$ with parameter $\beta$.  Suppose that Player A has a winning strategy in $\mathcal{G}'$, which they ``bring over'' to $\mathcal{G}$ as follows: after each move $B_j$ of Player B in game $\mathcal{G}$, Player A uses the biholomorphism $\phi$ to construct the ball $B'_j$ for game $\mathcal{G}'$, then makes the move $A'_j$ using the winning strategy, then uses $\phi^{-1}$ to construct a legal move $A_j$ for $\mathcal{G}$.  We claim that the assurance of Player A winning $\mathcal{G}'$ results in them winning $\mathcal{G}$ as well.

Let $B_1$ be the initial closed ball chosen by Player $B$ in game $\mathcal{G}$.  Set
\[ c_1:=\sup_{\bs{\xi}\in B_1}\|J_{\bs{\xi}}\|_{\mathrm{op}}, \quad c_2:=\sup_{\bs{\xi}'\in\phi(B_1)}\|J^{-1}_{\bs{\xi}'}\|_{\mathrm{op}}, \]
and choose $m$ large enough so that
\begin{equation}\label{Eq3.2}
 c(\beta+1)\beta^{m-2}<1,\qquad\text{where}~c:=2c_1c_2.
\end{equation}
Let $\beta'=\beta^m<\frac{1}{3}$.  Assume that, after sufficiently many moves, Player A can enforce
\[ \sup_{\bs{\eta},\bs{\eta}'\in\phi(B)}\dfrac{\|\widetilde{L}_{\bs{\eta}_0}(\bs{\eta})-\phi^{-1}(\bs{\eta})\|}{\|\bs{\eta}-\bs{\eta}_0\|}<2\beta' c_2\left(1+\dfrac{1}{\beta} \right) \]
corresponding to the move $B$ by Player B.  We relabel the index in such a way that the first such move $B$ is labeled as $B_1$.

Using the winning strategy for $\mathcal{G}'$, Player A inductively constructs moves \[ B'_j\supset B'_j\setminus A'_j \supset B'_{j+1}\supset\dots;\] at the same time, they construct a sequence of indices $1=j_1<j_2<\dots$ such that
\begin{equation}\label{Eq3.3}
 \phi(B_{j_k})\subset B'_k \quad\text{for~all}~k.
\end{equation}
This will imply that
\[ \phi\left( \bigcap_k B_k \right) = \phi\left( \bigcap_k B_{j_k} \right) = \bigcap_k \phi\left( B_{j_k} \right) \subset\bigcap_k B'_k\subset E, \]
so that $\bigcap_k B_k \in\phi^{-1}(E)$, and Player A wins $\mathcal{G}$.

To enable the above construction, we need to assume the following:
\begin{enumerate}[label=(\roman*)]
 \item The sets $A_j,B_j$ and $A'_j,B'_j$ are legal moves in games $\mathcal{G}$ and $\mathcal{G}'$, respectively,
 \item If \[ B_{j_k}=B(\mathbf{z}_k,\rho_k),B'_k=B(\phi(\mathbf{z}_k),\rho'_k) \quad\text{with}~\rho'_k=c_1\rho_k, \]
 then
 \begin{equation}\label{Eq3.4}
  \beta^m\leq\dfrac{\rho_k}{\rho_{k-1}}<\beta^{m-1}, \quad\text{where}~m~\text{is~as~in}~\eqref{Eq3.2}.
 \end{equation}
\item The moves $A'_k=\mathcal{H}_k'^{(\varepsilon'_k)}$ result from Player A's winning strategy in $\mathcal{G}'$, and
\[ A_{j_k}\supset\phi^{-1}\left( \mathcal{H}_k'^{(\alpha\rho'_k)} \right), \quad\text{where}~\alpha=\beta'\left( 1+\dfrac{1}{\beta} \right). \]
\end{enumerate}
Suppose that, with the above assumptions, we have made moves up to stage $l-1$.  Let Player A make arbitrary moves until Player B's move $B_\lambda$ satisfies $\dfrac{\rho(B)}{\rho_{l-1}}<\beta^{m-1}$, whereby set $j_l=\lambda$.  Since $\rho(B)\to0$, $j_l$ are well-defined; moreover, the choice of $j_l$ ensures that \eqref{Eq3.4} holds for $k=l$.  Now let $B'_l=B(\phi(\mathbf{z}_l),\rho'_l)$.  By \eqref{Eq3.4} and the definition of $\rho'_k$,
\begin{equation}\label{Eq3.5}
 \beta'=\beta^m\leq\dfrac{\rho_l}{\rho_{l-1}}=\dfrac{\rho'_l}{\rho'_{l-1}}<\beta^{m-1}=\dfrac{\beta'}{\beta}.
\end{equation}
Since $B_{j_l}\subset B_{j_{l-1}}$, we must have $\|\mathbf{z}_{l-1}-\mathbf{z}_l\|\leq \rho_{l-1}-\rho_l$; thus
\[ \|\phi(\mathbf{z}_{l-1})-\phi(\mathbf{z}_l)\|\leq c_1\|\mathbf{z}_{l-1}-\mathbf{z}_l\|\leq c_1(\rho_{l-1}-\rho_l)= \rho'_{l-1}-\rho'_l. \]
This shows that $B'_l\subset B'_{l-1}$.  Since $B_{j_l}\cap A_{j_{l-1}}=\emptyset$ and $A_{j_{l-1}}\supset \phi^{-1}\left( \mathcal{H}_{l-1}'^{(\alpha\rho'_{l-1})} \right)$, \eqref{Eq3.5} implies that
\[  \inf_{\mathbf{w}\in\mathcal{H}'_{l-1}}\|\phi(\mathbf{z}_l)-\mathbf{w} \| \geq\alpha\rho'_{l-1}=\beta'\left( 1+\dfrac{1}{\beta} \right)\rho'_{l-1}>\beta'\rho'_{l-1}+\rho'_l; \]
that is, $B'_l\cap\mathcal{H}_{l-1}'^{(\beta'\rho'_{l-1})}=\emptyset$, and so $B'_l\subset B'_{l-1}\setminus A'_{l-1}$.  Hence $B'_l$ is a legal move in $\mathcal{G}'$.

Choose $A'_l$ as per the winning strategy for $\mathcal{G}'$; say $A'_l=\mathcal{H}_l'^{(\varepsilon'_l)}$ for some $\varepsilon'_l\leq\beta'\rho'_l$.  Fix $\mathbf{w}'\in\mathcal{H}'_l$, and define $\mathcal{H}_l:=\widetilde{L}_{\mathbf{w}'}(\mathcal{H}'_l)$.  For any $\mathbf{w}\in B'_l$ with $\inf_{\bs{\omega}\in\mathcal{H}'_l}\|\mathbf{w}-\bs{\omega}\|<\alpha\rho'_l$, write $\mathbf{w}=\mathbf{w}_0+\widetilde{\mathbf{w}}$ with $\mathbf{w}_0\in\mathcal{H}'_l$.  The choice of the initial (relabeled) ball $B_1$ gives
\begin{align*} \|\phi^{-1}(\mathbf{w})-\widetilde{L}_{\mathbf{w}'}(\mathbf{w}_0))\| &\leq \|\phi^{-1}(\mathbf{w})-\phi^{-1}(\mathbf{w}_0)\|+\|\phi^{-1}(\mathbf{w}_0)-\widetilde{L}_{\mathbf{w}'}(\mathbf{w}_0))\| \\  &\leq 2c_2\alpha\rho'_l. \end{align*}
Setting
\[ \varepsilon_l:= 2c_2\alpha\rho'_l=c\alpha\rho_l, \]
we have
\[ A_{j_l}:=\mathcal{H}_l^{(\varepsilon_l)}\supset \phi^{-1}\left( \mathcal{H}_{l-1}'^{(\alpha\rho'_{l-1})} \right). \]
By \eqref{Eq3.2}, we have $c\alpha<\beta$, whereby $\varepsilon_l<\beta\rho_l$; that is, the move $A_{j_l}$ is legal in $\mathcal{G}$ and satisfies (iii) above.  This inductive step thus enables us to do the required construction of moves for Player A that wins the game $\mathcal{G}$ for them.
 \item  Let $(U,\varphi)$ and $(V,\psi)$ be admissible charts in the holomorphic systems on the complex manifolds $M$ and $N$, respectively, say with $U\subset\C^m$, $V\subset\C^n$.
 \[\begin{tikzcd}
U \arrow[dashed]{r}{\widetilde{\phi}} \arrow[swap]{d}{\varphi} & V \arrow{d}{\psi} \\
M \arrow{r}{\phi} & N
\end{tikzcd}
\] 
   By the definition of HAW sets on complex manifolds (cf. \S \ref{SecHAW}), $\varphi^{-1}(E)\cap U$ is HAW in $U$.  Using (1), $\widetilde{\phi}(\varphi^{-1}(E) \cap U)$ is HAW in $V$, and so $\psi\circ\widetilde{\phi}\circ\varphi^{-1}(E)=\phi(E)$ is HAW in $N$.
 \end{enumerate}
 \end{proof}

\section{Proof of the Main Theorem}
In view of the fact that finite covering maps preserve manifold structure, the complex version of \cite[Lemma 5.1]{AGGL} follows from Proposition \ref{L214.223} in {an} exactly analogous manner.
\begin{lemma}\label{L5.1}
 Let $\Gamma$ and $\Gamma'$ be mutually commensurable lattices in $G$.  Then for any subgroup $F$ of $G$, $E(F)$ is HAW on $G/\Gamma$ if and only if $E(F)$ is HAW on $G/\Gamma'$.
\end{lemma}

Before embarking on the proof of the main theorem, we record the $\mathrm{SL}_2(\C)$ analogue of \cite[Theorem 3.6]{KW17}:

\begin{proposition}\label{T1.1_single}
 Let $\Gamma$ be a lattice in $G=\mathrm{SL}_2(\C)$ and $F$ a one-parameter Ad-semisimple subgroup of $G$.  Then
 \[E(F):=\{x\in G/\Gamma\,:\,Fx~\text{is~bounded}\} \]
 is HAW.
\end{proposition}
\begin{proof}
 Following \cite{KW17}, let $\mathfrak{g}=\mathfrak{sl}_2(\C)\cong\C^3$.  Let $X=\mathrm{SL}_2(\C)$ be endowed with a holomorphic structure through an atlas $\{(U_x,\varphi_x):x\in X\}$ as follows: for $x\in X$, define
 \[ \mathrm{exp}_x:\mathfrak{g}\to X, \quad \mathbf{x}\mapsto \mathrm{exp}(x)\mathbf{x}. \]
 Let $W_x$ be a neighborhood of $\mathbf{0}\in\mathfrak{g}$ such that $\mathrm{exp}_x\big|_{W_x}$ is injective.  Define $U_x=\mathrm{exp}_x(W_x)$ and $\varphi_x=\mathrm{exp}_x^{-1}\big|_{W_x}$, so that $\{(U_x,\varphi_x):x\in X\}$ is an atlas on $X$ giving the required structure.
 
 Let $H^+=\left\{h_z=\begin{pmatrix}1 & z \\ 0 & 1\end{pmatrix} \right\}$ and $H^-=\left\{h_z=\begin{pmatrix}1 & 0 \\ z & 1\end{pmatrix} \right\}$ denote, respectively, the upper- and lower-triangular subgroups of $\mathrm{SL}_2(\C)$, with the respective Lie algebras $\mathfrak{h}^+$ and $\mathfrak{h}^-$, and let $\mathfrak{f}$ denote the Lie algebra of $F$.  Assume that, by a suitable composition of $\varphi_x$ with a biholomorphism of $W_x$, the connected-component foliation of $U_x$ is mapped to the foliation of $\mathfrak{g}$ by translates of $\mathfrak{p}=\mathfrak{h}^-\oplus\mathfrak{f}$.  Let $\pi:\mathfrak{g}\to\mathfrak{h}^+$ denote the projection from $\mathfrak{g}=\mathfrak{h}^+\oplus\mathfrak{p}$, and let $W_x^+=\pi(W_x)$.
 
Now consider a Schmidt game played on $\mathfrak{g}$ with Player B's first move $B_1\subset W_x$ for some $x\in X$, where Player A has a winning strategy in the corresponding Schmidt game on $\mathfrak{h}\cong\C$ with target set $\{z\in\C:F^+h_zx~\text{is~bounded}\}$ (cf. Proposition \ref{L214.223}).  After each move $B_j$ by B in the $\mathfrak{g}$-game, Player A takes $B'_j=\pi(B_j)$, constructs a closed disk $A'_j\in\mathfrak{h}$ using the winning strategy in the $\mathfrak{h}$-game, and then makes the move $A_j=\pi^{-1}(A'_j)$, which is a hyperplane neighborhood in $\mathfrak{g}$.  The assurance of Player A winning the $\mathfrak{h}$-game implies their winning the $\mathfrak{g}$-game as well.  Proposition \ref{P3.4} then implies the result.
\end{proof}

\begin{proof}[Proof of Theorem \ref{T1.1_AGGL}]
 Suppose that $\Gamma$ is a lattice in $G$, where $G$ is a product of (finitely many) copies of $\mathrm{SL}_2(\C)$.  By \cite[Theorem 5.22]{MSR}, $G$ decomposes as $G=\prod_{j=1}^kG_j$ where each $G_j$ is a product of copies of $\mathrm{SL}_2(\C)$, such that there are irreducible lattices $\Gamma_j\subset G_j$ with $\Gamma$ being commensurable with $\prod_{j=1}^k \Gamma_j$.  Lemma \ref{L5.1} then reduces this to the case where $\Gamma=\prod_{j=1}^k \Gamma_j$.  As the boundedness of an orbit in $G/\Gamma$ is equivalent to the boundedness of each of its projections in $G_j/\Gamma_j$, we can consider the case of $\Gamma$ being irreducible and non-cocompact using Lemma \ref{L2.2}(4).  By the Margulis arithmeticity theorem \cite[Chapter IX, Theorem 1.9A]{Mar}, $\Gamma$ is commensurable with $\mathbf{G}(\Z[i])$ for a $\Q$-simple semisimple group $\mathbf{G}$; that is, $\Gamma$ is arithmetic.  Furthermore, for some totally imaginary field $K$, $\mathbf{G}=\mathrm{Res}_{K/\Q}\mathbf{G}'$, where $\mathbf{G}'$ is a $K$-form on $\mathrm{SL}_2$.  The noncompactness of $\Gamma$ implies that $\mathbf{G}'$ is $K$-isotropic; hence $\mathbf{G}'=\mathrm{SL}_{\mathbf{2}}$ and $\Gamma$ is commensurable with $\mathrm{Res}_{K/\Q}\mathrm{SL}_2(\Z[i])$.  Now let $F=\{g_t:t\in\R\}$ be a one-parameter Ad-semisimple subgroup of $G$.  By the Jordan-Chevalley decomposition, $F=F_{\mathrm{diag}}F_{\mathrm{bdd}}$, where elements of $F_{\mathrm{diag}}$ are diagonalizable over $\C$ and those of $F_{\mathrm{bdd}}$ are bounded.  It follows that $E(F)=E(F_{\mathrm{diag}})$, and there are $h\in G$ and $\mathbf{r}$ for which $F_{\mathrm{diag}}=hF_{\mathbf{r}}h^{-1}$.  Therefore $E(F_{\mathrm{diag}})=hE(F_{\mathbf{r}})$, and by Proposition \ref{L214.223}(2) and Proposition \ref{P3.1}, $E(F)$ is HAW.
\end{proof}

\section{An application to badly approximable complex numbers}
As an application of our work, we study the analogue of the set $\mathbf{Bad}(\Q(\sqrt{-D}))$ for a {totally imaginary} number field $K$.  Set
\begin{equation}\label{badk}
 \mathbf{Bad}(K)=\left\{ z\in\C:\inf_{p,q\in\mathcal{O}_K,~q\neq 0} \left\{|q|\left|qz+p \right| \right\}>0 \right\}.
\end{equation}
This definition is equivalent to the one given in \eqref{1_Eq7.1}, as remarked in \cite[\S1]{KL}.

We first obtain a description of $\mathbf{Bad}(K,\mathbf{r})$ in terms of the set $\mathbf{Bad}(K)$ for a special vector $\mathbf{r}$.
\begin{lemma}\label{L7.1}
 Let $K$ be a totally imaginary number field of degree $n$ over $\Q$. Let $e_1=(1,0,\dots,0)\in\R^{n}$.  Then
 \[ \mathbf{Bad}(K,e_1)=\{(\sigma_1(z),\dots,\sigma_{n}(z))\in\C^{n}:z\in \mathbf{Bad}(K)\}, \]
 where $\sigma_j\in\Sigma$ are field embeddings of $K$ into $\C$ with $\sigma_1:z\mapsto z$. 
\end{lemma}
\begin{proof}
 Let $\{b_1,b_2,\dots,b_{n}\}$ be the standard integral basis of $\mathcal{O}_K$ with $b_1=1$ so that each $z\in\C$ (respectively, $p\in K$) can be uniquely written as $z=\sum_{j=1}^n z_jb_j$ (respectively, $p=\sum_{j=1}^n p_jb_j$), where $z_j\in\R$ (respectively, $p_j\in\Q$).  Note that the embeddings $\sigma_j$ of $K\hookrightarrow\C$ are precisely of the form
 \[ \sigma_j(z) = \sum_{\underset{\scriptstyle{k\neq j}}{k=1}}^{n} z_kb_k + z_j\overline{b_j}, \]
 where $\overline{\phantom{O}}$ denotes the usual complex conjugation.  Consider $\mathbf{r}=e_1=(1,0,\dots,0)\in\R^{n}$.  Then, in the notation of \S\ref{S3}, we have $\Sigma_+=\{\sigma_1\}$ and $\Sigma_0=\{\sigma_2,\dots,\sigma_{n}\}$.  In view of \S6, \eqref{badk}, Definition \ref{D3.2} and Proposition \ref{P3.4}, we see that
 \begin{align*}
  z\in\mathbf{Bad}(K) &\iff \inf_{\overset{\scriptstyle p\in\mathcal{O}_{K}}{q\in\mathcal{O}_{K}\setminus\{0\}}}\{|q||qz+p|\}>0 \\
 &\iff \inf_{\overset{\scriptstyle p\in\mathcal{O}_{K}}{q\in\mathcal{O}_{K}\setminus\{0\}}}\max\{|q||qz+p|,\max_{2\leq j \leq n}\{|\sigma_j(q)\sigma_j(z)+\sigma_j(p)|,|\sigma_j(q)|\}\}>0 \\
  &\iff (\sigma_1(z),\dots,\sigma_{n}(z))\in\mathbf{Bad}(K,e_1).
 \end{align*}
\end{proof}
We are now set to obtain our main application using the above description.
\begin{theorem}\label{T7.1}
 Let $K$ be as in Lemma \ref{L7.1}.  Then $\mathbf{Bad}(K)$ has full Hausdorff dimension.
\end{theorem}
\begin{proof}
 By Proposition \ref{P3.5}, $\mathbf{Bad}(K,e_1)$ is HAW.  By Lemmas \ref{L2.1}(1) and \ref{L2.2}(1), this set is $\frac{1}{2}$-winning and has full Hausdorff dimension ($=2n$).  In view of Lemma \ref{L7.1}, note that $\mathbf{Bad}(K)=\Pi_1(\mathbf{Bad}(K,e_1))\subset\C$ is the natural projection in the first coordinate, so that
 \begin{equation}\label{hdim1}
\dim_H\mathbf{Bad}(K)\leq\dim_H\C=2.
\end{equation}
We claim equality in \eqref{hdim1}.  Indeed, if $\dim_H\mathbf{Bad}(K)<2$, then \[ \dim_H(\mathbf{Bad}(K)\times\C^{n-1})=\dim_H\mathbf{Bad}(K) + \dim_H\C^{n-1}<2+2(n-1)=2n. \]
But $\mathbf{Bad}(K,e_1)\subset\mathbf{Bad}(K)\times\C^{n-1}$, so that
\[ 2n=\dim_H\mathbf{Bad}(K,e_1)\leq\dim_H(\mathbf{Bad}(K)\times\C^{n-1})<2n, \]
a contradiction.  Thus, equality holds in \eqref{hdim1}; that is, $\mathbf{Bad}(K)$ has full Hausdorff dimension.
\end{proof}

\section*{Acknowledgements}
I sincerely thank Prof. Anish Ghosh for suggesting the question, for helpful discussions and for guidance throughout the course of this work.  I also thank Prof. Erez Nesharim and the anonymous referees for helpful comments.  One of the referees  of a previous version of this paper suggested a modification of the proof of Lemma \ref{L4.4} which facilitated proving Theorem \ref{T1.1_AGGL} and results in \S7 in the generality stated in this version.  This work is supported by CSIR-JRF grant 09/877(0014)/2019-EMR-I.

\end{document}